\def\anonymous{0}

\documentclass[runningheads]{llncs}

\usepackage[utf8]{inputenc}
\usepackage{amsmath,amssymb}
\usepackage{yhmath}
\usepackage{tikz}
\usepackage{standalone}
\usepackage{subcaption}
\usepackage{enumitem}
\usetikzlibrary{calc,angles,quotes,perspective}

\usepackage[T1]{fontenc}

\usepackage{graphicx}

\usepackage[hidelinks]{hyperref}

\usepackage{color}

\usepackage{cleveref}

\newcommand{\bbR}{\mathbb{R}}
\newcommand{\bbM}{\mathbb{M}}

\newcommand{\dif}{\textrm{d}}

\DeclareMathOperator{\E}{E}

\DeclareMathOperator{\GL}{GL}

\DeclareMathOperator{\im}{im}

\let\act\vartriangleright

\begin{document}

\title{Universal Collection of Euclidean Invariants between Pairs of Position-Orientations}
\titlerunning{Universal Collection of \(\E(3)\) Invariants on \(\bbM_3 \times \bbM_3\)}

{\if\anonymous0
    \author{
        Gijs Bellaard\and
        Bart M. N. Smets\and
        Remco Duits
    }
    \authorrunning{G. Bellaard et al.}

    \institute{
        CASA \& EAISI, Eindhoven University of Technology \\
        \email{\{g.bellaard,b.m.n.smets,r.duits\}@tue.nl}
    }
\else
    \author{Anonymous Author(s)}
    \authorrunning{Anonymous Author(s)}
    \institute{Anonymous Institute(s)}
\fi}

\maketitle 

\begin{abstract} 
Euclidean \(\E(3) := \bbR^3 \rtimes \text{O}(3)\) equivariant neural networks that employ scalar fields on position-orientation space \(\bbM_3 := \bbR^3 \times S^2\) have been effectively applied to tasks such as predicting molecular dynamics and properties.
To perform equivariant convolutional-like operations in these architectures one needs Euclidean invariant kernels \(k : \bbM_3 \times \bbM_3 \to \bbR\).
In practice, a handcrafted collection of invariants is selected, and this collection is then fed into multilayer perceptrons to parametrize the kernels.
We rigorously describe an optimal collection of \(4\) smooth scalar invariants on the whole of \(\bbM_3 \times \bbM_3\).
With optimal we mean that the collection is \emph{independent} and \emph{universal}, meaning that \emph{all} invariants are pertinent, and \emph{any} invariant kernel is a function of them.
We evaluate two collections of invariants, one universal and one not, using the PONITA neural network architecture.
Our experiments show that using a collection of invariants that is universal positively impacts the accuracy of PONITA significantly.

\keywords{Position-Orientation Space \and Euclidean Group \and Differential Geometry \and Lie Theory \and Invariant \and Equivariant \and Machine Learning.}

\end{abstract}

\section{Introduction}

Neural networks that are equivariant with respect to the Euclidean group \(\E(3) := \bbR^3 \rtimes \text{O}(3)\) -- that being the Lie group of all translations, rotations, and reflections of \(\bbR^3\) -- have been successfully applied to various tasks: 
1) point-cloud classification \cite{fuchs2020se3}, where an object and its rigid transformations should be classified as the same, 
2) molecular dynamics/properties \cite{bekkers2024fast,fuchs2020se3,gasteiger2021gemnet,liu2022spherical}, as the (non-relativistic) physical laws are invariant under Euclidean transformations, 
and 3) hemodynamics \cite{suk2024mesh,suk2024labgatr}, where, for example, wall shear stress vector prediction should be equivariant w.r.t. \(\E(3)\).

There are many ways to create Euclidean equivariant neural networks, but here we focus on one specific kind: models that use scalar fields \(f : \bbM_3 \to \bbR\) on position-orientation space \(\bbM_3 := \bbR^3 \times S^2\) as their feature maps.
The choice is well supported since: 1) in \cite[Sec.3]{gasteiger2021gemnet} it is shown theoretically that scalar fields \(\bbM_3\) provide the same expressivity as more mathematically involved networks that use \(\rho : \text{O}(3) \to \GL(V)\) representation fields \(f : \bbR^3 \to V\) \cite{anderson2019cormorant,fuchs2020se3,thomas2018tensor,weiler2018steerable}, and 2) the recent state-of-the-art results in \cite{bekkers2024fast} show that scalar fields on \(\bbM_3\) can be equally expressive practically.

Suppose we have a scalar field \(f : \bbM_3 \to \bbR\) as input and want to construct a linear operator to process it, as is common in neural networks.
Consider an integral operator \(\Phi\) of the form
\begin{equation}
    (\Phi f)(p_1) = \int_{\bbM_3} k(p_1,p_2) f(p_2) \dif \mu(p_2),
\end{equation}
where \(k : \bbM_3 \times \bbM_3 \to \bbR\) is a so-called kernel, and \(\mu\) is a (\(\E(3)\) invariant) measure on position-orientation space \(\bbM_3\).
To make the mapping \(\Phi\) Euclidean equivariant it is sufficient to enforce the following invariance constraint on the kernel \(k\):
\begin{equation}
    k(g \act p_1, g \act p_2) = k(p_1, p_2) \text{ for all } p_1, p_2 \in \bbM_3, g \in \E(3).
\end{equation}
Here \(\act\) is the standard action of $\E(3)$ on a position and orientation in \(\bbM_3\) and is formally defined later in \Cref{sec:prelim}.

Hence, to create \(\E(3)\) equivariant neural networks on \(\bbM_3\) we are motivated to study scalar \(\E(3)\) invariants \(\iota : \bbM_3 \times \bbM_3 \to \bbR\), that being functions with the property that \(\iota(g \act p_1, g \act p_2) = \iota(p_1, p_2)\) for all \(p_1, p_2 \in \bbM_3\) and \(g \in \E(3)\).

Consider any collection of scalar invariants \(\iota_1, \dots, \iota_n : \bbM_3 \times \bbM_3 \to \bbR\).
We can create a new invariant \(\iota'\) easily by considering any function \(h : \bbR^n \to \bbR\) and defining \(\iota' = h(\iota_1, \dots, \iota_n)\).
This observation has an immediate application in our neural networks: we can decide to parameterize the kernels \(k\) by, for example, a multi-layer perceptron \(\text{MLP}_\theta : \bbR^n \to \bbR\) with (trainable) parameters \(\theta\), and inserting a predesigned collection of \(n\) invariants: \(k = \text{MLP}_{\theta}(\iota_1, \dots, \iota_n)\).
This motivates looking into what an ``optimal'' collection of scalar invariants would be, so that we can construct networks that are maximally expressive and efficient.

Suppose we have a collection of invariants where one of them is a function of the others.
If this happens we say the collection of invariants is \emph{dependent} (\Cref{def:dependent}).
A dependent collection is not ``optimal'' in the sense that we could remove the dependent invariant and (theoretically) lose no expressiveness.

On the other hand, suppose we have a collection of invariants for which we know that any other invariant one can think of is a function of them.
We say such a collection of invariants is \emph{universal} (\Cref{def:universality_representer}).
A universal collection of invariants is ``optimal'' in the sense that there is no reason to add another invariant because we (theoretically) gain no expressiveness.

This motivates our quest for a collection of scalar invariants that is both independent and universal, which will serve as a type of basis for invariant kernels.

In \cite{bekkers2024fast} Bekkers et al. present PONITA, an \(\E(3)\) equivariant architecture that utilizes scalar fields on \(\bbM_3\).
This architecture achieves the previously mentioned state-of-the-art results on two molecular datasets: rMD17 \cite{christensen2020role}, where the task is to predict molecular dynamics, and QM9 
\cite{ramakrishnan2014quantum,ruddigkeit2012enumeration}, where the goal is to predict chemical properties of various molecules.
PONITA utilizes the following collection of three \(\E(3)\) invariants \(\iota_i : \bbM_3 \times \bbM_3 \to \bbR\):
\begin{equation} \label{eq:invariants_ponita}
\begin{split}
    \iota_1(p_1, p_2) &= (x_2 - x_1) \cdot n_1, \\
    \iota_2(p_1, p_2) &= \|(x_2 - x_1) - \iota_1(p_1, p_2) n_1 \|, \\
    \iota_3(p_1, p_2) &= \arccos(n_1 \cdot n_2),
\end{split}
\end{equation}
where \(p_1 = (x_1, n_1)\), \(p_2 = (x_2, n_2) \in \bbM_3 = \bbR^3 \times S^2\). 

However, we noticed that these invariants are not universal, see \Cref{rem:bekkers_not_complete}, meaning that not all kernels can be written in terms of this collection of invariants, which could impact the accuracy of PONITA negatively in certain applications.
Therefore we theoretically investigate in detail what an independent and universal collection of scalar \(\E(3)\) invariants on \(\bbM_3 \times \bbM_3\) would be, and if such a collection can be used to improve the accuracy of PONITA in practice.

\subsubsection{Contributions}
In \Cref{def:our_invariants} we give our own collection of four smooth \(\E(3)\) invariants \(\iota_i : \bbM_3 \times \bbM_3 \to \bbR\), of which we show in \Cref{res:our_invariants_are_cool} that it is both independent and universal.
In \Cref{sec:experiments} we show that using a universal collection of invariants has a significant positive impact on the accuracy of PONITA when predicting molecular properties.

\subsubsection{Outline}
In \Cref{sec:prelim} we briefly define our main objects of study; that being the position-orientation space \(\bbM_3\), the Euclidean group \(\E(3)\), and invariants.
In \Cref{sec:universality} we define universality and the existence of a \emph{representer} (\Cref{def:universality_representer}), and show that the latter implies the former.
In \Cref{sec:dependence} we define (in)dependence and show that invariants that form a submersion are independent.
In \Cref{sec:our_invariants} we present a collection of four smooth scalar invariants, along with proofs of its independence and universality.
In \Cref{sec:experiments} we evaluate our universal collection of invariants using the PONITA architecture.
In \Cref{sec:conclusion} we conclude the paper.

\section{Preliminaries} \label{sec:prelim}

Next we briefly define our central concepts: the position-orientation space \(\bbM_3\), the Euclidean group \(\E(3)\), how \(\E(3)\) acts on \(\bbM_3\), and what an invariant is.

\begin{definition}[Position-Orientation Space] \label{def:pos_ori_space}
    The smooth manifold of three-dimensional position-orientations is
    \( \bbM_3 = \{ (x,n) \in \bbR^3 \times \bbR^3 \mid \|n\|=1 \}\).
    The tangent space at a point \(p = (x,n) \in \bbM_3\) is 
    \( T_p \bbM_3 = \{ (\dot x, \dot n) \in \bbR^3 \times \bbR^3 \mid \dot n \cdot n = 0 \}\).
\end{definition}

\begin{definition}[Euclidean Group]
    The Euclidean group is
    \(\E(3) = \{ (t,Q) \in \bbR^3 \times \bbR^{3 \times 3} \mid  Q^\top Q = I\} \).
    The group product is
    \( (t_2, Q_2) \cdot (t_1, Q_1) = (t_2 + Q_2 t_1, Q_2 Q_1) \).
    The identity element is \(e = (0,I)\).
\end{definition}

We define the action \(\act : \E(3) \times \bbM_3 \to \bbM_3\) of the Euclidean group on position-orientation space by 
\( (t,Q) \act (x, n) = (t + Q x, Q n),\) 
where \((t,Q) \in \E(3)\) and \((x,n) \in \bbM_3\).
We extend this action to pairs of position-orientations \((p_1, p_2) \in \bbM_3 \times \bbM_3\) by letting the group element act on both, that is \(g \act (p_1, p_2) := (g \act p_1, g \act p_2)\).
From here on out we will drop the group action symbol \(\act\) for conciseness.

The following defines what an invariant is for arbitrary sets \(X\) and groups \(G\). 
While this level of generality exceeds our needs -- since our focus will be on \(X=\bbM_3×\times \bbM_3\) and \(G=\E(3)\) -- we present the full definition here for the sake of completeness.

\begin{definition}[Invariant]
    Let \(X\) be a set and \(G\) a group acting on it.
    Define \(\sim\) as the equivalence relation indicating if two elements are in the same \emph{orbit}, that is \(x \sim x' \Leftrightarrow \exists g \in G : g x = x'\).
    Let \(Y\) be any set.
    An \emph{invariant} \(\iota : X \to Y\) is a mapping such that \(x \sim x' \Rightarrow \iota(x) = \iota(x')\).
\end{definition}


\section{Universality} \label{sec:universality}

Here, we will define what universality is and clarify what it means for an invariant to have a representer. 
We will demonstrate in \Cref{res:representer_universal_complete} that the existence of a representer implies universality, a result we will later use to prove that our collection of invariants \eqref{eq:invariants_ours} is universal.

\begin{definition}[Universality and Representer] \label{def:universality_representer}
    \begin{itemize}
        \item An invariant \(\iota : X \to Y\) is \emph{universal} if any other invariant \(j : X \to Z\) can be written as a function of it, that is there exists \(F : \im \iota \to Z\) such that \(j = F \circ \iota\).
        \item A \emph{representer} \(\varphi\) for an invariant \(\iota : X \to Y \) is a map \(\varphi : \im \iota \to X\) such that \(\varphi(\iota(x)) \sim x\) for all \(x \in X\).
    \end{itemize}
\end{definition}

In other words, a representer \(\varphi\) takes in the value \(\iota(x)\) of the invariant \(\iota\) of a \(x \in X\) and returns a \(x' = \varphi(\iota(x))\) that is in the orbit of \(x\).

\begin{lemma} \label{res:representer_universal_complete}
    Let \(\iota\) be an invariant. 
    If \(\iota\) has a representer then \(\iota\) is universal.
\end{lemma}

\begin{proof} 
    Let \(\varphi\) be a representer for \(\iota\) and let \(j\) be any other invariant.
    Now consider \(F = j \circ \varphi\).
    One can check that \(j = F \circ \iota\), thus showing that \(\iota\) is universal.
    \qed
\end{proof}

\begin{remark}
    Even with a universal collection of invariants, we cannot perfectly represent arbitrary invariants in practice, as expressivity is constrained by the chosen architecture. 
    For instance, if our invariants are continuous and we use multi-layer perceptrons with continuous activation functions to parametrize the kernels, we can only represent continuous invariants, meaning we exclude any with discontinuities.
\end{remark}

\begin{remark} \label{rem:bekkers_not_complete}
    We can show that the invariants \eqref{eq:invariants_ponita} used in \cite{bekkers2024fast} are not universal as follows.
    Let \(e_1, e_2, e_3\) be the standard basis of \(\bbR^3\).
    Consider the following two pairs of position-orientations: \((p_1,p_2) = ((0,e_3),(e_1,e_2))\) and \((q_1,q_2) = ((0,e_3),(e_1,e_1))\).
    As per \eqref{eq:invariants_ponita}, we calculate their invariants to be
    \(\iota_1(p_1,p_2) = \iota_1(q_1,q_2) = 0\), \(\iota_2(p_1,p_2) = \iota_2(q_1,q_2) = 1\), and \(\iota_3(p_1,p_2) = \iota_3(q_1,q_2) = \pi/2\).
    We see that all invariants agree between the two pairs.
    Consider the invariant \(\iota = (x_2-x_1) \cdot n_2\).
    We calculate that \(\iota(p_1,p_2)=0\) and \(\iota(q_1,q_2)=1\), so \(\iota\) can not be a function of \(\iota_1\), \(\iota_2\), and \(\iota_3\).
    Thereby, we conclude that \eqref{eq:invariants_ponita} is not an universal collection of invariants.
\end{remark}

\section{Dependence} \label{sec:dependence}

In this section, we will define what it means for a collection of invariants to be (in)dependent.
We will show in \Cref{res:submersion} that a collection of invariants that form a submersion is independent, a result we will later use to prove that our collection of invariants \eqref{eq:invariants_ours} is independent.

\begin{definition}[Dependent] \label{def:dependent}
    A collection of scalar invariants \(\iota_1, \dots, \iota_n : X \to \bbR\) is \emph{dependent} over \(X\) if one of them can be written as a function of the others on \(X\), and \emph{independent} otherwise.
\end{definition}

If a collection of invariants is dependent on a set \(X\) then it is also dependent on any subset \(X' \subseteq X\).
Contrapositively, if a collection of invariants is independent on a set \(X'\) it is also independent on any superset \(X \supseteq X'\).

This observation raises an issue in our definition of independence: it is easy to create a ``large'' set on which the invariants are technically independent; one only needs to include a ``small'' subset on which this is the case.  
To remedy this we introduce the following stronger definition.

\begin{definition}[Somewhere Dependent]
    Let \(X\) be a topological space and \(\iota : X \to \bbR^n\) a collection of scalar invariants.
    We say \(\iota\) is \emph{somewhere dependent} if there exists an open subset \(U \subseteq X\) such that \(\iota\) is dependent on \(U\), and \emph{everywhere independent} otherwise.
\end{definition} 

It might seem hard to prove that a collection of scalar invariants is everywhere independent. 
However, if we specialize our setting by considering spaces that are differentiable manifolds and invariants that are differentiable mappings, we have the following simpler sufficient condition.

\begin{lemma} \label{res:submersion}
    Let \(M\) be a differentiable manifold and \(\iota : M \to \bbR^n\) a collection of differentiable scalar invariants that form a \emph{submersion}, that is the differential \(d\iota|_p : T_pM \to \bbR^n\) is a surjective linear map at every point \(p \in M\).
    Then \(\iota\) is everywhere independent.
\end{lemma}

\begin{proof} 
    A submersion is an \emph{open map} \cite[Prop.4.28]{lee2012introduction}, meaning that it maps open sets to open sets.
    So, given any arbitrary open subset \(U \subseteq M\) the image \(\iota(U) \subseteq \bbR^n\) is open.
    There is no functional relation between all elements of an open subset in \(\bbR^n\) (such an open set is never the graph of a function), thus \(\iota\) is independent on \(U\).
    As \(U\) was arbitrary, this shows that \(\iota\) is everywhere independent. 
    \qed
\end{proof}

\begin{remark}
    In practice, constructing a network using a dependent collection of invariants poses no issues, as the redundancy does not hinder training and may even enhance stability.
    However, proving the independence of a universal collection remains valuable, as it implies that removing \emph{any} invariant renders the collection \emph{non}-universal: the invariant that is removed can not be written as a function of the remaining invariants.
\end{remark}

\section{A Universal and Independent Collection of Invariants} \label{sec:our_invariants}

\begin{definition}[The Invariants] \label{def:our_invariants}
    Write \(p_1=(x_1,n_1)\), \(p_2=(x_2,n_2) \in \bbM_3\).
    We define the following smooth functions \(\iota_i : \bbM_3 \times \bbM_3 \to \bbR\):
    \begin{equation} \label{eq:invariants_ours}
    \begin{array}{lll}
        \iota_1(p_1,p_2) &= (x_2 - x_1) \cdot n_1, \\
        \iota_2(p_1,p_2) &= (x_2 - x_1) \cdot n_2,\\
        \iota_3(p_1,p_2) &= (x_2 - x_1) \cdot (x_2 - x_1), \\ 
        \iota_4(p_1,p_2) &= n_1 \cdot n_2.
    \end{array}
    \end{equation}
\end{definition}

We will prove that 1) the mappings \(\iota_i\) are invariants, 2) they form a universal collection, and 3) that they are independent everywhere.
These results are collected in the following theorem.

\begin{theorem} \label{res:our_invariants_are_cool}
    The functions \eqref{eq:invariants_ours} form a universal and everywhere independent collection of \(\E(3)\) invariants on \(\bbM_3 \times \bbM_3\).
\end{theorem}

\begin{proof} 
    We will begin by proving that the functions are invariants, then demonstrate their universality, and finally show that they are independent everywhere.

    \textbf{Invariance:}
    To prove that these are invariants we must show that \(\iota(p_1, p_2)=\iota(g p_1, g p_2)\) for all pairs of position-orientations \((p_1, p_2) \in \bbM_3 \times \bbM_3\) and all rigid transformations \(g=(t,Q) \in \E(3)\).
    This is a straightforward calculation using that \((Q x) \cdot (Q y) = x \cdot y\) 
    for all \(x,y \in \bbR^3\) and orthogonal matrices \(Q \in \text{O}(3)\).

    \textbf{Universality:}
    Let \((p_1=(x_1,n_1)\), \(p_2=(x_2,n_2) \in \bbM_3\) be a pair of position-orientations. 
    From the values of the invariants \(\iota_1(p_1, p_2), \dots, \iota_4(p_1,p_2)\) we will construct another pair of position-orientations \((\bar p_1, \bar p_2)\) such that \((\bar p_1, \bar p_2) \sim (p_1,p_2)\). 
    The described construction will act as our representer, thus showing that the collection of invariants is universal, as per  \Cref{res:representer_universal_complete}.

    Consider the vectors \(v_1 = n_1, v_2 = n_2, v_3 = x_2 - x_1\). 
    Their Gram matrix is
    \begin{equation}
        v_i \cdot v_j = \left[\begin{matrix}
            1 & \iota_4 & \iota_1 \\
            \iota_4 & 1 & \iota_2 \\
            \iota_1 & \iota_2 & \iota_3 
        \end{matrix}\right]_{ij}, \textrm{ with }\iota_k=\iota_k(p_1,p_2), 
    \end{equation}
    meaning that we have access to this Gram matrix purely from the values of the invariants.
    Let \(\bar v_1, \bar v_2, \bar v_3 \) be \emph{any} list of vectors with the same Gram matrix.
    Lists of vectors with the same Gram matrix are related by an orthogonal matrix \(Q\), i.e. \(Q \bar v_i = v_i\).
    Now pick \emph{any} position \(\bar x_1\), and define \(\bar n_1 = \bar v_1\), \(\bar n_2 = \bar v_2\), and \(\bar x_2 = \bar x_1 + \bar v_3\).
    We check that the rigid transformation \(g = (- Q \bar x_1 + x_1, Q) \in \E(3)\) maps \((\bar p_1, \bar p_2)\) to \((p_1, p_2)\), showing that \((\bar p_1, \bar p_2) \sim (p_1, p_2)\).  For example
    \begin{equation}
    \begin{split}
        g \act \bar{p}_2 
        &= (- Q \bar x_1 + x_1 + Q \bar{x}_2, Q \bar{n}_2) 
        = (Q(\bar{x}_2 - \bar{x}_1) + x_1, Q \bar{n}_2)\\
        &= (Q \bar{v}_3 + x_1, Q \bar{v}_2)
        = (v_3 + x_1, v_2)
        = ((x_2-x_1)+x_1, n_2)
        = p_2,
    \end{split}
    \end{equation}
    and similarly one finds $g \act \bar{p}_1=p_1$.
    
    \textbf{Independence:}
    To prove that this collection of invariants is everywhere independent we will show that the differentials \(d \iota_i|_p : T_p(\bbM_3 \times \bbM_3) \to \bbR\) are linearly independent on the following dense and open subset \(U\) of \(\bbM_3 \times \bbM_3\):
    \begin{equation}
        U := \{(p_1, p_2) \in \bbM_3 \times \bbM_3 \mid x_2 - x_1, \ n_1, \text{ and } n_2 \text{ form a basis of } \bbR^3\},
    \end{equation}
    where \(p_1=(x_1,n_1)\), \(p_2=(x_2,n_2) \in \bbM_3\).
    This shows that \(d \iota|_p\) is surjective everywhere on \(U\), meaning that \(\iota\) is a submersion from \(U\) to \(\bbR^4\), and therefore everywhere independent, as per \Cref{res:submersion}.  
    In short, it suffices to show that $d \iota_i|_p$ are independent on \(U\), which we will do next.
    
    Let \(\dot p_1=(\dot x_1, \dot n_1) \in T_{p_1}\bbM_3\), \(\dot p_2=(\dot x_2, \dot n_2) \in T_{p_2}\bbM_3\).
    The differentials \(d\iota_i|_p\) are
    \begin{equation}
    \begin{split}
        d\iota_1|_p(\dot p_1, \dot p_2) &= (\dot x_2 - \dot x_1) \cdot n_1 + (x_2 - x_1) \cdot \dot n_1, \\
        d\iota_2|_p(\dot p_1, \dot p_2) &= (\dot x_2 - \dot x_1) \cdot n_2 + (x_2 - x_1) \cdot \dot n_2, \\
        d\iota_3|_p(\dot p_1, \dot p_2) &= 2 (\dot x_2 - \dot x_1) \cdot (x_2 - x_1) , \\
        d\iota_4|_p(\dot p_1, \dot p_2) &= \dot n_1 \cdot n_2 + n_1 \cdot \dot n_2.
    \end{split}
    \end{equation}
    Suppose that we could find coefficients \(c^i \in \bbR\) such that \(\sum_{i=1}^4 c^i d \iota_i|_p = 0\) (that being for \emph{all} tangents \(\dot p_1\), \(\dot p_2\)).
    Specifically, take \(\dot x_1 = \dot n_1 = \dot n_2 = 0\).
    We find \((c^1 n_1 + c^2 n_2 + 2 c^3 (x_2 - x_1) ) \cdot \dot x_2 = 0\) for \emph{all} \(\dot x_2\).
    This implies that \(c^1 n_1 + c^2 n_2 + 2 c^3 (x_2 - x_1) = 0\).
    Because we have assumed that these three vectors form a basis, we get that \(c^1 = c^2 = c^3 = 0\).
    Now consider \(\dot x_1 = \dot x_2 = \dot n_2 = 0\).
    We find \((c^1 (x_2 - x_1) + c^4 n_2) \cdot \dot n_1 = 0\) for \emph{all} \(\dot n_1\) with \(\dot n_1 \cdot n_1 = 0\) (See \Cref{def:pos_ori_space}).
    This implies that \(c^1 (x_2 - x_1) + c^4 n_2\) is in the span of \(n_1\).
    Again, because we have assumed that these three vectors form a basis, we have to conclude that \(c^1 = c^4 = 0\).
    All in all, we find that \(\sum_{i=1}^4 c^i d \iota_i|_p = 0\) implies that all \(c^i = 0\), thus showing that the differentials are linearly independent.
    \qed
\end{proof}

\section{Experiments} \label{sec:experiments}

By \Cref{rem:bekkers_not_complete}, the invariants \eqref{eq:invariants_ponita} used in \cite{bekkers2024fast} are not universal, which may limit PONITA's expressivity and  accuracy in specific applications. 
In contrast, the invariants we propose \eqref{eq:invariants_ours} are universal, suggesting that, in theory, a performance difference between the two collections of invariants should be observable.


For this reason we performed some experiments on the QM9 dataset \cite{ramakrishnan2014quantum,ruddigkeit2012enumeration}, specifically predicting chemical properties of various molecules (134k stable small organic molecules).
We choose to discretize with \(16\) orientations, use \(6\) layers, \(128\) dimensional features, and train for \(800\) epochs.
All other model settings and hyperparameters are kept the same as in \cite[App.E.2]{bekkers2024fast}.
We report the mean absolute error (MAE) on the test set with the model that did best on the validation set during training.
The results can be found in \Cref{tab:ponita_qm9}.

\definecolor{good}{RGB}{54, 176, 60}
\definecolor{bad}{RGB}{176, 74, 110}

We see that the universal collection of invariants we propose \eqref{eq:invariants_ours} outperforms the non-universal invariants suggested in \eqref{eq:invariants_ponita} on 10 of the 12 targets we experimented on, with an average improvement of \(\color{good}-14.4\%\).

On some targets (\(\langle R^2 \rangle\), ZPVE, and \(U_0\)) the change in accuracy is less pronounced, in two cases even deteriorating.
One explanation is that a universal collection may be unnecessary for these targets, meaning the Bekkers et al. invariants are already sufficiently expressive. 
The remaining variability in accuracy can be attributed to initialization and training.

On other targets (\(U\), \(H\), and \(G\)) the improvement in accuracy is stark.
We can identify this as a direct consequence of using a universal collection of invariants: these targets cannot be accurately predicted using the incomplete invariants of Bekkers et al.

\begin{table}[ht]
    \centering
    \def\arraystretch{1.2} 
    \setlength{\tabcolsep}{0.5em} 
    \begin{tabular}{ll|rrr}
        Target & Unit & Bekkers et al. \eqref{eq:invariants_ponita} & Universal (Ours) \eqref{eq:invariants_ours} & Difference \%\\
        \hline
        \(\mu\) & D & 0.0195 & 0.0166 & \(\color{good}-15.0\)\\
        \(\alpha\) & \(a_0^3\) & 0.0557 & 0.0489 & \(\color{good}-12.1\)\\
        \(\varepsilon_\text{homo}\) & eV & 0.0226 & 0.0202 & \(\color{good}-10.4\)\\
        \(\varepsilon_\text{lumo}\) & eV & 0.0206 & 0.0187 & \(\color{good}-9.0\)\\
        \(\Delta \varepsilon\) & eV & 0.0415 & 0.0378 & \(\color{good}-8.9\)\\
        \(\langle R^2 \rangle\) & \(a_0^2\) & 0.4160 & 0.4251 & \(\color{bad}+2.2\)\\
        ZPVE & meV & 1.5647 & 1.5241 & \(\color{good}-2.6\)\\
        \(U_0\) & eV & 0.9920 & 1.0285 & \(\color{bad}+3.7\)\\
        \(U\) & eV & 1.3593 & 0.7362 & \(\color{good}-45.8\)\\
        \(H\) & eV & 1.0205 & 0.6934 & \(\color{good}-32.1\)\\
        \(G\) & eV & 1.1856 & 0.7721 & \(\color{good}-34.9\)\\
        \(c_v\) & \(\frac{\mathrm{cal}}{\mathrm{mol} \, \mathrm{K}}\) & 0.0292 & 0.0270 & \(\color{good}-7.4\)
    \end{tabular}
    \caption{
        PONITA trained to predict chemical properties of various molecules (QM9 dataset).
        Mean absolute error on the test set is reported (lower is better).
        Our universal invariants perform better than the invariants used in \cite{bekkers2024fast}.
    }
    \label{tab:ponita_qm9}
\end{table}

\section{Conclusion} \label{sec:conclusion}

In \Cref{def:our_invariants} we introduced a collection of four smooth \(\E(3)\) invariants \(\iota_i : \bbM_3 \times \bbM_3 \to \bbR\).
We proved in \Cref{res:our_invariants_are_cool} that this collection is both independent and universal, meaning that all invariants are pertinent, and any other invariant is a function of them.

In \Cref{sec:experiments} we performed an experimental comparison between our collection of invariants \eqref{eq:invariants_ours} and those proposed in Bekkers et al. \cite{bekkers2024fast}, as defined in \eqref{eq:invariants_ponita}. 
We did this by training the PONITA architecture to predict chemical properties of various molecules (QM9 dataset \cite{ramakrishnan2014quantum,ruddigkeit2012enumeration}).
We observe improvements in accuracy as a result of using a universal set of invariants.

\textbf{Availability of Code.} 
All code can be found at 
\if\anonymous0
\url{https://gitlab.com/gijsbel/ponita_invariants}.
\else
\url{https://anonymized.url}.
\fi

\textbf{Acknowledgments.} 
{\if\anonymous0
The Dutch Research Council (NWO) is gratefully acknowledged for financial support via VIC.202.031 (Duits, Geometric learning for Image Analysis \url{https://www.nwo.nl/en/projects/vic202031}).
The European Union is gratefully acknowledged for financial support through project REMODEL (Horizon Europe, MSCA-SE, 101131557 \url{https://doi.org/10.3030/101131557}).
\fi}
We thank Bekkers et al. \cite{bekkers2024fast} for their publicly available PONITA architecture \url{https://github.com/ebekkers/ponita}.

\textbf{Addendum.} 
We thank Johan Edstedt for bringing \cite[Eq.2,Prop.1]{deng2018ppffold} to our attention. 
It presents the same invariants and a similar proof of their universality. 
Unfortunately, we were not aware of this paper at the time of publication.


%
%
%
%

\bibliographystyle{splncs04}
\bibliography{bibliography.bib}

\newpage

\end{document}